\documentclass[11pt]{article}
\usepackage[latin9]{inputenc}
\usepackage{color}
\usepackage{xcolor}
\usepackage{algorithm2e}
\usepackage{mathtools}
\usepackage{amsmath}
\usepackage{amssymb}
\usepackage[authoryear]{natbib}
\usepackage{geometry}

\makeatletter
\usepackage{graphicx}
\usepackage{amsfonts}
\providecommand{\U}[1]{\protect\rule{.1in}{.1in}}
\providecommand{\U}[1]{\protect\rule{.1in}{.1in}}
\newtheorem{theorem}{Theorem}
\newtheorem{algo}{Algorithm}

\newtheorem{definition}{Definition}

\newtheorem{lemma}{Lemma}
\newtheorem{problem}{Problem}

\newtheorem{remark}{Remark}
\newenvironment{proof}[1][Proof]{\textbf{#1.} }{\ \rule{0.5em}{0.5em}}

\makeatother

\usepackage[parfill]{parskip}
\usepackage[nogroupskip]{glossaries-extra}

\makeglossaries

\newacronym{rttp}{RTTP}{Radiation therapy treatment planning}
\newacronym{pvc}{PVC}{Percentage violation constraint}
\newacronym{dvc}{DVC}{Dose volume constraint}
\newacronym{moscfpp}{MOSCFPP}{Multiple-operator split common fixed point problem}
\newacronym{sip}{SIP}{Split inverse problem}
\newacronym{scfp}{SCFP}{Split convex feasibility problem}
\newacronym{imrt}{IMRT}{Intensity modulated radiation therapy}
\newacronym{msscfp}{MSSCFP}{Multiple sets split convex feasibility problem}
\newacronym{cmsscfp}{CMSSCFP}{Constrained multiple set split convex feasibility problem}
\newacronym{scfpp}{SCFPP}{Split common fixed points problem}
\newacronym{fne}{FNE}{Firmly nonexpansive}
\newacronym{sqne}{SQNE}{Strongly quasi-nonexpansive}
\newacronym{msa}{MSA}{Modular string averaging}
\newacronym{aca}{ACA}{Almost cyclic sequential algorithm}
\newacronym{cfp}{CFP}{Convex feasibility problem}
\newacronym{lfp}{LFP}{Linear feasibility problem}
\newacronym{oar}{OAR}{Organ at risk}
\newacronym{ptv}{PTV}{Planning target volume}
\newacronym{hdc}{HDC}{Hard dose constraint}
\newacronym{dvh}{DVH}{Dose-volume histogram}

\usepackage[absolute,overlay]{textpos}
\usepackage{url}

\begin{document}

\begin{titlepage}

\title{\vspace{-0.7in}\textbf{Dynamic string-averaging CQ-methods for the split feasibility
problem with percentage violation constraints arising in radiation
therapy treatment planning}}
\author{Mark Brooke$^{1}\footnote{Corresponding author: mark.brooke@oncology.ox.ac.uk, cc: markdanielbrooke@gmail.com}$, Yair Censor$^{2}$ and Aviv Gibali$^{3}\bigskip$\\
$^{1}$Department of Oncology, University of Oxford, \\ Oxford OX3 7DQ, UK\\
\{mark.brooke@oncology.ox.ac.uk\} 
\\
$^{2}$Department of Mathematics, University of Haifa\\
Mt.\ Carmel, Haifa 3498838, Israel \\
\{yair@math.haifa.ac.il\} 
\\
$^{3}$Department of Mathematics, ORT Braude College \\
Karmiel 2161002, Israel \\
\{avivg@braude.ac.il\} 
}
\date{January 14, 2020. Revised: December 18, 2020.}

\maketitle
\begin{abstract}
In this paper we study a feasibility-seeking problem with percentage
violation constraints. These are additional constraints, that are
appended to an existing family of constraints, which single out certain
subsets of the existing constraints and declare that up to a specified
fraction of the number of constraints in each subset is allowed to
be violated by up to a specified percentage of the existing bounds.
Our motivation to investigate problems with percentage violation constraints
comes from the field of radiation therapy treatment planning wherein
the fully-discretized inverse planning problem is formulated as a
split feasibility problem and the percentage violation constraints
give rise to non-convex constraints. Following the CQ algorithm of Byrne (2002, \textit{Inverse Problems}, Vol. \textbf{18}, pp. 441--53), we develop a string-averaging
CQ method that uses only projections onto the individual sets which
are half-spaces represented by linear inequalities. The question of
extending our theoretical results to the non-convex sets case is still
open. We describe how our results apply to radiation therapy treatment
planning and provide a numerical example.
\end{abstract}

\begin{textblock*}{19cm}(1.2cm,25.3cm) % {block width} (coords)
\scriptsize\textcolor{gray}{
%\textbf{Full Open Access version available:} \url{https://doi.org/10.1111/itor.12929}\\ \\
This is the peer reviewed version of the following article: Brooke, M., Censor, Y. and Gibali, A. (2021), Dynamic string-averaging CQ-methods for the split feasibility problem with percentage violation constraints arising in radiation therapy treatment planning. \textit{Intl. Trans. in Op. Res.}, which has been published in final form at \url{https://doi.org/10.1111/itor.12929}. This article may be used for non-commercial purposes in accordance with Wiley Terms and Conditions for Use of Self-Archived Versions.}
\end{textblock*}

\end{titlepage}

\section{Introduction\label{sect:intro}}

\subsection{Motivation\label{subsec:motivation}}

In this work we are motivated by a linear split feasibility problem
with percentage violation constraints arising in radiation therapy
treatment planning. We first provide the background in general
terms.

\textbf{Inverse radiation therapy treatment planning (RTTP)}. This
problem, in its fully-discretized modeling approach, leads to a linear
feasibility problem. This is a system of linear interval inequalities
\begin{equation}
c\leq Ax\leq b,\label{eq:interval}
\end{equation}
wherein the ``dose matrix''\ $A$ is pre-calculated by techniques
called in RTTP ``forward calculation''\ or ``forward planning''\ and
the vector $x$ is the unknown vector of ``intensities''\ that,
when used in setting up the treatment machine, will realize this specific
``treatment plan''. The vectors $b$ and $c$ contain upper and
lower bounds on the total dose $Ax$ permitted and required in volume
elements (voxels) of sensitive organs/tissues and target areas, respectively,
inside the irradiated body. The components of $b$ and $c$ are prescribed
by the attending oncologist and given to the treatment planner.

\textbf{Percentage violation constraints (PVCs)}. In general terms,
these are additional constraints that are appended to an existing
family of constraints. They single out certain subsets of the existing
constraints and declare that up to a specified fraction of the number
of constraints in each subset is allowed to be violated by up to a
specified percentage of the existing bounds. Such PVCs are useful
in the inverse problem of RTTP, mentioned above, where they are called
``dose volume constraints'' (DVCs). When the system of linear interval
inequalities is inconsistent, that is, there is no solution vector
that satisfies all inequalities, the DVCs allow the oncologist and
the planner to relax the original constraints in a controlled manner
to achieve consistency and find a solution.

\textbf{Split feasibility}. PVCs are, by their very nature, integer
constraints which change the feasibility problem to which they are
attached from being a continuous feasibility problem into becoming
a mixed integer feasibility problem. An alternative to the latter
is to translate the PVCs into constraints sets that are appended to
the original system of linear interval inequalities but are formulated
on the vectors $Ax,$ rather than directly on $x.$ This gives rise
to a ``split feasibility problem''\ which is split between two
spaces: the space of ``intensity vectors''\ $x$ and the space
of ``dose vectors''\ $d:=Ax$ in which $A$ is the operator mapping
one space onto the other.

\textbf{Non-convexity}. The constraints sets, that arise from the
PVCs, in the space of ``dose vectors''\ are non-convex sets but,
due to their special form enable the calculation of orthogonal projections
of points onto them. This opens the door for applying our proposed
dynamic string-averaging CQ-method to the RTTP inverse problem with
PVCs. Mathematical analysis for the case of non-convex sets remains
an open question. Looking at it from the practical point of view one
may consider also alternatives such as reformulating PVCs as $\ell_{1}$-norm
constraints. See, for example, \citet{candes08,kim13}.

\textbf{Group-structure of constraints}. Each row in the system (\ref{eq:interval})
represents a constraint on a single voxel. Lumping together constraints
of voxels, according to the organ/tissue to which they belong, divides
the matrix $A$ and the whole system into ``groups''\ of constraints,
referred to below as ``blocks of constraints''\ in a natural manner.
These groups affect the formulation of the split feasibility problem
at hand by demanding that the space of intensity vectors\ $x$ be
mapped separately by each group of rows of the matrix $A$ into another
space of dose vectors\ $d.$

\subsection{Contribution\label{subsec:contrib}}

Motivated by the above we deal in this paper with the ``multiple-operator
split common fixed point problem''\ (MOSCFPP) defined next.

\begin{problem} \label{prob:moscfpp}\textbf{The multiple-operator
split common fixed point problem\ (MOSCFPP)}. Let $\mathcal{H}$
and $\mathcal{K}$ be two real Hilbert spaces, and let $r$ and $p$
be two natural numbers. Let $U_{i}:\mathcal{H}\rightarrow\mathcal{H}$,
$1\leq i\leq p,$ and $T_{j}:\mathcal{K}\rightarrow\mathcal{K},$
$1\leq j\leq r,$ be given operators with nonempty fixed point sets
$\text{\ensuremath{\mathrm{Fix}}}\left(U_{i}\right)$ and $\text{\ensuremath{\mathrm{Fix}}}\left(T_{j}\right)$,
respectively. Further, let $A_{j}:\mathcal{H}\rightarrow\mathcal{K},$
for all $1\leq j\leq r,$ be given bounded linear operators. In addition
let{{} $\Phi$} be another closed and convex subset
of $\mathcal{H}$. The MOSCFPP is:
\begin{align}
{\text{\text{Find an }\ensuremath{x^{\ast}\in\Phi\text{ such that }}}}x^{\ast}\in\cap_{i=1}^{p}\text{\ensuremath{\mathrm{Fix}}}(U_{i})\text{ and, }\\
\text{for all }1 & \leq j\leq r,\text{ }A_{j}x^{\ast}\in\text{\ensuremath{\mathrm{Fix}}}(T_{j}).\label{eq:prob1-line2}
\end{align}

\end{problem}

This problem formulation unifies several existing ``split\ problems''\ formulations
and, to the best of our knowledge, has not been formulated before.
We analyze it and propose a ``dynamic string-averaging CQ-method''\ to
solve it, based on techniques used in some of those earlier formulations.
We show in detail how this problem covers and applies to the linear
split feasibility problem with DVCs in RTTP. Our convergence results
about the dynamic string-averaging CQ-algorithm presented here rely
on convexity assumptions. Therefore, there remains an open question
whether our work can be expanded to cover the case of the non-convex
constraints in the space of dose vectors\ $d$ used in RTTP$.$ Recent
work in the field report on strides made in the field of projection
methods when the underlying sets are non-convex, {see,
for example, \citet{hln14,bpw14,abs13}}. This encourages us to expand
the results presented here in the same way.

\subsection{Structure of the paper\label{subsec:structure}}

We begin by briefly reviewing relevant ``split problem'' formulations
which have led to our proposed MOSCFPP and a ``dynamic string-averaging
CQ-method'' to solve it. Starting from a general formulation of two
concurrent inverse problems in different vector spaces connected by
a bounded linear operator, we explore the inclusion of multiple convex
constraint sets within each vector space. Defining operators that
act on each of these sets allows us to formulate equivalent fixed
point problems, which naturally leads to our MOSCFPP. We then provide
some insight into how one may solve such a problem, using constrained
minimization, or successive metric projections as part of a CQ-type
method \citep{byrne1}. These projection methods form the basis of
our ``dynamic string-averaging CQ-method'', which is introduced
in Section \ref{Sec:Method}. Important mathematical foundations for
this method are provided in Section \ref{Sec:Prelims}, which serve
to describe the conditions under which the method converges to a solution
in Section \ref{subsec:convergence}. Finally, we bring percentage
violation constraints (PVCs) into our problem formulation (Section
\ref{sect:PVC}) and consolidate our work by providing examples of
how the MOSCFPP and ``dynamic string-averaging CQ-method'' may be
applied in RTTP (Section \ref{sec:application}). {A numerical example is provided on a synthetically created treatment plan, detailed in Section \ref{sec:numerical}}.

{An important comment must be made here. The introduction
of a new mathematical model for an application naturally calls for
simulated numerical validation, particularly when a new algorithm
is proposed. Here we present a rudimentary numerical example since more complex clinically-relevant treatment plans rely heavily on the medical physics context of the radiation therapy treatment planning problem. As such, they call for evaluation of the results in the context of the radiation therapy treatment planning problem itself and require a dedicated proper background and framework which are outside the scope of this
paper. An extensive analysis of the methods presented in this paper, on a number of clinical treatment plans, will be published in an appropriate medical physics journal.}

\section{A brief review of ``split\ problems''\ formulations and solution
methods\label{sec:review}}

The following brief review of ``split\ problems''\ formulations
and solution methods will help put our work in context. The review
is non-exhaustive and focuses only on split\ problems that led to
our new formulation that appears in Problem \ref{prob:moscfpp}. Other
split problems such as {``the common solution of the split variational inequality problems and fixed point problems''}\ \citep[see, e.g.,][]{Lohawech18}
or ``split Nash equilibrium problems for non-cooperative strategic
games''\ \citep[see, e.g.,][]{Jinlu-Li-19} and many others are
not included here. The ``split inverse problem''\ (SIP), which
was introduced by \citet{cgr12} \citep[see also][]{bcgr12}, is formulated
as follows.

\begin{problem} \label{prob:split-inverse}\textbf{The split inverse
problem\ (SIP)}. Given are two vector spaces $X$ and $Y$ and a
bounded linear operator $A:X\rightarrow Y$. In addition, two inverse
problems are involved. The first one, denoted by IP$_{1}$, is formulated
in the space $X$ and the second one, denoted by IP$_{2}$, is formulated
in the space $Y$. The SIP is:
\begin{equation}
\text{Find an }x^{\ast}\in X\text{ that solves IP}_{1}\text{ such that }y^{\ast}:=Ax^{\ast}\in Y\text{ solves IP}_{2}\text{.}
\end{equation}

\end{problem}

The first published instance of a SIP is the ``split convex feasibility
problem'' (SCFP) of \citet{CE94}, which is formulated as follows.

\begin{problem} \label{prob:split-feas}\textbf{The split convex
feasibility problem (SCFP)}. Let $\mathcal{H}$ and $\mathcal{K}$
be two real Hilbert spaces. Given are nonempty, closed and convex
sets $C\subseteq\mathcal{H}$ and $Q\subseteq\mathcal{K}$ and a bounded
linear operator $A:\mathcal{H}\rightarrow\mathcal{K}$. The SCFP is:
\begin{equation}
\text{Find an }\ x^{\ast}\in C\ \text{such that}\ Ax^{\ast}\in Q.\label{sfp}
\end{equation}

\end{problem}

This problem was employed, among others, for solving an inverse problem
in intensity-modulated radiation therapy (IMRT) treatment planning
{\citep[see][]{cbmt06,davidi15,CEKB2004}}. More results regarding the SCFP theory and algorithms, can be found, for example, in \citet{Yang04,lopez12,glt17},
and the references therein. The SCFP was extended in many directions
to Hilbert and Banach spaces formulations. It was extended also to
the following ``multiple sets split convex feasibility problem''
(MSSCFP).

\begin{problem} \label{prob:multisetscfp}\textbf{The multiple sets
split convex feasibility problem\ (MSSCFP)}. Let $\mathcal{H}$ and
$\mathcal{K}$ be two real Hilbert spaces and $r$ and $p$ be two
natural numbers. Given are sets $C_{i},$ $1\leq i\leq p$ and $Q_{j},$
$1\leq j\leq r,$ that are closed and convex subsets of $\mathcal{H}$
and $\mathcal{K}$, respectively, and a bounded linear operator $A:\mathcal{H}\rightarrow\mathcal{K}$.
The MSSCFP is:
\begin{equation}
\text{Find an }x^{\ast}\in\cap_{i=1}^{p}C_{i}\text{ such that }Ax^{\ast}\in\cap_{j=1}^{r}Q_{j}.\label{P:MSSCFP}
\end{equation}

\end{problem}

\citet{ms07} proposed the ``constrained multiple set split convex
feasibility problem'' (CMSSCFP) which is phrased as follows (see also
\citet{lve16}).

\begin{problem} \label{prob:massad}\textbf{The constrained multiple
set split convex feasibility problem (CMSSCFP)}. Let $\mathcal{H}$
and $\mathcal{K}$ be two real Hilbert spaces and $r$ and $p$ be
two natural numbers. Given are sets $C_{i},$ $1\leq i\leq p$ and
$Q_{j},$ $1\leq j\leq r,$ that are closed and convex subsets of
$\mathcal{H}$ and $\mathcal{K}$, respectively, and for $1\leq j\leq r$,
given bounded linear operators $A_{j}:\mathcal{H}\rightarrow\mathcal{K}$.
In addition let ${\Phi}$ be another closed and convex
subset of $\mathcal{H}$. The CMSSCFP is:
\begin{equation}
\text{Find an }x^{\ast}\in{\Phi}\text{ such that }x^{\ast}\in\cap_{i=1}^{p}C_{i}\text{ and }A_{j}x^{\ast}\in Q_{j},\text{ for }1\leq j\leq r.
\end{equation}

\end{problem}

Another extension, due to \citet{CS09}, is the following ``split
common fixed points problem'' (SCFPP).

\begin{problem} \label{prob:scommonfp}\textbf{The split common fixed
points problem (SCFPP)}. Let $\mathcal{H}$ and $\mathcal{K}$ be
two real Hilbert spaces and $r$ and $p$ be two natural numbers.
Given are operators $U_{i}:\mathcal{H}\rightarrow\mathcal{H}$, $1\leq i\leq p,$
and $T_{j}:\mathcal{K}\rightarrow\mathcal{K},$ $1\leq j\leq r,$
with nonempty fixed point sets $\mathrm{\mathrm{Fix}}\left(U_{i}\right),$
$1\leq i\leq p$ and $\mathrm{\mathrm{Fix}}\left(T_{j}\right)$, $1\leq j\leq r$,
respectively, and a bounded linear operator $A:\mathcal{H}\rightarrow\mathcal{K}$.
The SCFPP is:
\begin{equation}
\text{Find an }x^{\ast}\in\cap_{i=1}^{p}\mathrm{Fix}(U_{i})\text{ such that }Ax^{\ast}\in\cap_{j=1}^{r}\text{\ensuremath{\mathrm{Fix}}}(T_{j}).
\end{equation}

\end{problem}

Problems \ref{prob:split-feas}--\ref{prob:scommonfp} are SIPs but,
more importantly, they are special cases of our MOSCFPP of Problem
\ref{prob:moscfpp}.

Focusing in a telegraphic manner on algorithms for solving some of
the above SIPs, we observe that the SCFP of Problem \ref{prob:split-feas}
can be reformulated as the constrained minimization problem:
\begin{equation}
\min_{x\in C}\frac{1}{2}\Vert P_{Q}(Ax)-Ax\Vert^{2},\label{P:cm}
\end{equation}
where $P_{Q}$ is the orthogonal (metric) projection onto $Q$. {(Note that the term ``orthogonal projection" is used mainly for subspaces while the ``metric" projection refers to any kind of sets \citep[see, e.g., ][Section 2.2.4]{Cegielski12})}. Since the objective function is convex and continuously differentiable with Lipschitz continuous gradients, one can apply the projected gradient
method \citep[see, e.g.,][]{goldstein64} and obtain Byrne's well-known
\textit{CQ-algorithm} \citep{byrne1}. The iterative step of the CQ-algorithm
has the following structure:
\begin{equation}
x^{k+1}=P_{C}(x^{k}-\gamma A^{\star}(Id-P_{Q})Ax^{k}),
\end{equation}
where $A^{\star}$ stands for the adjoint ($A^{\star}$=$A^{T}$ transpose
in Euclidean spaces) of $A$, $\gamma$ is some positive number, $Id$
is the identity operator, and $P_{C}$ and $P_{Q}$ are the orthogonal
projections onto $C$ and $Q,$ respectively. For the MSSCFP of Problem
\ref{prob:multisetscfp}, the minimization model considered in \citet{CEKB2004},
is
\begin{equation}
\min_{x\in\mathbb{R}^{M}}\left(\sum_{i=1}^{p}\text{dist}^{2}(x,C)+\sum_{j=1}^{r}\text{dist}^{2}(Ax,Q)\right),\label{P:min-multi-scfp}
\end{equation}
leading, for example, to a gradient descent method which has an iterative
simultaneous projections nature:
\begin{equation}
x^{k+1}=x^{k}-\gamma\sum_{i=1}^{p}\alpha_{i}\left(Id-P_{C_{i}}\right)x^{k}+\sum_{j=1}^{r}\beta_{j}A^{\star}\left(Id-P_{Q_{j}}\right)Ax^{k},\label{MSSCFP:step}
\end{equation}
where $\gamma\in\left(0,\frac{2}{L}\right)$ with
\begin{equation}
L:=\sum_{i=1}^{p}\alpha_{i}+\sum_{j=1}^{r}\beta_{j}\Vert A\Vert_{F}^{2}
\end{equation}
where $\Vert A\Vert_{F}^{2}$ is the squared Frobenius norm of $A$.\\

Inspired by the above and the work presented in \citet{pzcbcs17},
we propose in the sequel a ``dynamic string-averaging CQ-method''\ for
solving the MOSCFPP of Problem \ref{prob:moscfpp}.

\section{Preliminaries\label{Sec:Prelims}}

Through this paper $\mathcal{H}$ and $\mathcal{K}$ are two real
Hilbert spaces and let $D\subset\mathcal{H}$. For every point $x\in\mathcal{H}$,
there exists a unique nearest point in $D$, denoted by $P_{D}(x)$
such that
\begin{equation}
\Vert x-P_{D}(x)\Vert\leq\Vert x-y\Vert,\ \text{for all\ }y\in D.
\end{equation}

The operator $P_{D}:\mathcal{H}\rightarrow\mathcal{H}$ is called
the \textit{metric projection} onto $D$.

\begin{definition} Let $T:\mathcal{H}\rightarrow\mathcal{H}$ be
an operator and $D\subset\mathcal{H}$.

(i) The operator $T$ is called \texttt{Lipschitz continuous} on $D$
with constant $L>0$ if
\begin{equation}
\Vert T(x)-T(y)\Vert\leq L\Vert x-y\Vert,\text{\ for all\ }x,y\in D.
\end{equation}

(ii) The operator $T$ is called \texttt{nonexpansive} on $D$ if
it is $1$-Lipschitz continuous.

(iii) The \texttt{Fixed Point set} of $T$ is
\begin{equation}
\mathrm{Fix}(T):=\{x\in\mathcal{H}\mid T(x)=x\}.
\end{equation}

(iv) The operator $T$ is called $c$-\texttt{averaged} ($c$-av)
\citep{bbr-1} if there exists a nonexpansive operator $N:D\rightarrow\mathcal{H}$
and a number $c\in(0,1)$ such that
\begin{equation}
T=(1-c)Id+cN.
\end{equation}
In this case we also say that $T$ is $c$-av \citep{byrne2}. If
two operators $T_{1}$ and $T_{2}$ are $c_{1}$-av and $c_{2}$-av,
respectively, then their composition $S=T_{1}T_{2}$ is $(c_{1}+c_{2}-c_{1}c_{2})$-av;
\citep[see][Lemma 2.2.]{byrne2}

(v) The operator $T$ is called $\nu$\texttt{-inverse strongly monotone}
($\nu$-ism) on $D$ if there exists a number $\nu>0$ such that
\begin{equation}
\langle T(x)-T(y),x-y\rangle\geq\nu\Vert T(x)-T(y)\Vert^{2},\text{ for all }x,y\in D.
\end{equation}

(vi) The operator $T$ is called \texttt{firmly nonexpansive} (FNE)
on $D$ if
\begin{equation}
\left\langle T(x)-T(y),x-y\right\rangle \geq\left\Vert T(x)-T(y)\right\Vert ^{2},\text{ for all }x,y\in D\text{,}
\end{equation}

A useful fact is that $T$ is firmly nonexpansive if and only if its
complement $Id-T$ is firmly nonexpansive. Moreover, $T$ is firmly
nonexpansive if and only if $T$ is $(1/2)$-av (see \citet[Proposition 11.2]{GR84}
and \citet[Lemma 2.3]{byrne2}). In addition, $T$ is averaged if
and only if its complement $Id-T$ is $\nu$-ism for some $\nu>1/2$;
\citep[see, e.g.,][Lemma 2.1]{byrne2}.\\

(vii) The operator $T$ is called \texttt{quasi-nonexpansive }(QNE)
\begin{equation}
\Vert T(x)-w\Vert\leq\Vert x-w\Vert\text{ for all }(x,w)\in\mathcal{H}\times\mathrm{Fix}(T)
\end{equation}

(viii) The operator $T$ is called is called a \texttt{cutter} (also
\texttt{firmly quasi-} \texttt{nonexpansive}) ($T\in\mathfrak{T}$)
if $\mathrm{Fix}(T)\neq\emptyset$ and
\begin{equation}
\left\langle T(x)-x,T(x)-w\right\rangle \leq0\text{\ for all }(x,w)\in\mathcal{H}\times\mathrm{Fix}(T).
\end{equation}

(ix) Let $\lambda\in\lbrack0,2]$, the operator $T_{\lambda}:=(1-\lambda)Id+\lambda T$
is called $\lambda$-\texttt{relaxation} of the operator $T$. With
respect to cutters above it is known that for $\lambda\in\lbrack0,1]$,
the $\lambda$-relaxation of a cutter is also a cutter \citep[see, e.g.,][Remark 2.1.32]{Cegielski12}.\\

(x) The operator $T$ {is called} $\rho$-\texttt{strongly
quasi-nonexpansive} ($\rho$-SQNE), where $\rho\geq0$, if $\mathrm{Fix}(T)\neq\emptyset$
and
\begin{equation}
\Vert T(x)-w\Vert\leq\Vert x-w\Vert-\rho\Vert T(x)-x\Vert,\text{\ for all }(x,w)\in\mathcal{H}\times\mathrm{Fix}(T).
\end{equation}
A useful fact is that a family of SQNE operators with non-empty intersection
of fixed point sets is closed under composition and convex combination
\citep[see, e.g.,][Corollary 2.1.47]{Cegielski12}.\\

(xi) The operator $T$ is called is called \texttt{demi-closed} at
$y\in\mathcal{H}$ if for any sequence $\left\{ x^{k}\right\} _{k=0}^{\infty}$
in $D$ such that $x^{k}\rightarrow\overline{x}\in D$ and $T(x^{k})\rightarrow y,$
we have $T(\overline{x})=y.$ \end{definition}

Next we recall the well-known \textit{Demi-closedness Principle} \citep{Browder}.

\begin{lemma} \label{Lem:DM} Let $\mathcal{H}$ be a Hilbert space,
$D$ a closed and convex subset of $\mathcal{H}$, and $N:D\rightarrow\mathcal{H}$
a nonexpansive operator. Then $Id-N$ ($Id$ is the identity operator
on $\mathcal{H}$) is \texttt{demi-closed} at $y\in\mathcal{H}.$
\end{lemma}

Let $A:\mathcal{H}\rightarrow\mathcal{K}$ be a bounded linear operator
with $\|A\|>0$, and $C\subseteq\mathcal{H}$ and $Q\subseteq\mathcal{K}$
{be} nonempty, closed and convex sets. The operator $V:\mathcal{H}\rightarrow\mathcal{H}$
which is defined by
\begin{equation}
V:=Id-\frac{1}{\Vert A\Vert^{2}}A^{\star}(Id-P_{Q})A\label{Eq:Lan}
\end{equation}
is called a \textit{Landweber operator} and $U:\mathcal{H}\rightarrow\mathcal{H}$
defined by
\begin{equation}
U:=P_{C}{V}
\end{equation}
is called a \textit{projected Landweber operator} {with
$V$ as in (\ref{Eq:Lan})}. See, e.g., \citet{Cegielski12,cegielski15,cegielski14}.

In the general case where $T:\mathcal{H}\rightarrow\mathcal{H}$ is
quasi-nonexpansive and $A:\mathcal{H}\rightarrow\mathcal{K}$ is a
bounded and linear operator with $\Vert A\Vert>0$, {a
so-called }\textit{{Landweber-type operator}}{{}
\citep[see, e.g.,][]{cegielski14} is defined by}
\begin{equation}
V:=Id-\frac{1}{\Vert A\Vert^{2}}A^{\star}(Id-T)A.\label{Eq:LanType}
\end{equation}
{Note that (\ref{Eq:Lan}) is a special case of (\ref{Eq:LanType}),
since $P_{Q}$ is firmly nonexpansive, thus, quasi-nonexpansive.}

\section{The dynamic string-averaging CQ-method\label{Sec:Method}}

In this section we present our ``dynamic string-averaging CQ-method''\ for
solving the MOSCFPP of Problem \ref{prob:moscfpp}. It is actually
an algorithmic scheme which encompasses many specific algorithms that
are obtained from it by different choices of strings and weights.
First, for all $j=1,2,\ldots,r,$ construct from the given data of
Problem \ref{prob:moscfpp}, the operators $V_{j}:\mathcal{H}\rightarrow\mathcal{H}$
defined by
\begin{equation}
V_{j}:=Id-\gamma_{j}A_{j}^{\star}(Id-T_{j})A_{j},\label{eq:V_j}
\end{equation}
where $\gamma_{j}\in\left(0,\frac{1}{L_{j}}\right)$, $L_{j}=\Vert A_{j}\Vert^{2}$.
For quasi-nonexpansive $T_{j}$ this definition coincides with that
of ``Landweber-type operators related to $T_{j}$''\ of \citet[Definition 2]{cegielski14}
with a relaxation of $\gamma_{j}$.

For simplicity, and without loss of generality, we assume that $r=p$
in Problem \ref{prob:moscfpp}. This is not restrictive since if $r<p$
we will define $T_{j}:=Id$ for $r+1\leq j\leq p,$ and if $p<r$
we will define $U_{i}:=Id$ for $p+1\leq i\leq r,$ which, in both
cases, will not make any difference to the formulation of Problem
\ref{prob:moscfpp}.

Define $\Gamma:=\{1,2,\dots,p\}$ and for each $i\in\Gamma$ define
the operator $R_{i}:\mathcal{H}\rightarrow\mathcal{H}$ by$R_{i}:=U_{i}V_{i}.$
An \textit{index vector} is a vector $t=(t_{1},t_{2},\dots,t_{q})$
such that $t_{i}\in\Gamma$ for all $i=1,2,\dots,q$. For a given
index vector $t=(t_{1},t_{2},\dots,t_{q})$ we denote its \textit{length}
by $\ell(t):=q,$ and define the operator $Z[t]$ as the product of
the individual operators $R_{i}$ whose indices appear in the index
vector $t$, namely,
\begin{equation}
Z[t]:=R_{t_{\ell(t)}}R_{t_{\ell(t)-1}}\cdots R_{t_{1}},\label{eq:Zt}
\end{equation}
and call it a \textit{string operator}. A finite set $\Theta$ of
index vectors is called \textit{fit} if for each $i\in\Gamma$, there
exists a vector $t=(t_{1},t_{2},\dots,t_{q})\in\Theta$ such that
$t_{s}=i$ for some $s\in\Gamma$.

Denote by $\mathcal{M}$ the collection of all pairs $(\Theta,w)$,
where $\Theta$ is a fit finite set of index vectors and
\begin{equation}
w:\Theta\rightarrow(0,\infty)\text{ is such that }\sum_{t\in\Theta}w(t)=1.\label{eq:1.6}
\end{equation}

For any $(\Theta,w)\in\mathcal{M}$ define the convex combination
of the end-points of all strings defined by members of $\Theta$ by
\begin{equation}
{\Psi_{\Theta,w}(x)}:=\sum_{t\in\Theta}w(t)Z[t](x),\;x\in\mathcal{H}.\label{eq:1.7}
\end{equation}
We fix a number $\Delta\in(0,1/p)$ and an integer $\bar{q}\geq p$
and denote by $\mathcal{M}_{\ast}\equiv\mathcal{M}_{\ast}(\Delta,\bar{q})$
the set of all $(\Theta,w)\in\mathcal{M}$ such that the lengths of
the strings are bounded and the weights are all bounded away from
zero, namely,
\begin{equation}
\mathcal{M}_{\ast}:=\{(\Theta,w)\in\mathcal{M\mid}\text{ }\ell(t)\leq\bar{q}\text{ and }w(t)\geq\Delta\text{ for all }t\in\Theta\}.\label{eq:11516}
\end{equation}

The dynamic string-averaging CQ-method with variable strings and variable
weights is described by the following iterative process.

\begin{algo}\label{alg:sa-cq}\textbf{The dynamic string-averaging
CQ-method with variable strings and variable weights}

\textbf{Initialization}: Select an arbitrary $x^{0}\in\mathcal{H}$,

\textbf{Iterative step}: Given a current iteration vector $x^{k}$
pick a pair $(\Theta_{k},w_{k})\in\mathcal{M}_{\ast}$ and calculate
the next iteration vector by
\begin{equation}
x^{k+1}={\Psi}_{{\Theta_{k},w_{k}}}(x^{k})\text{.}\label{eq:algv}
\end{equation}
\end{algo}

The iterative step of (\ref{eq:algv}) amounts to calculating, for
all $t\in\Theta_{k},$ the strings' end-points
\begin{equation}
Z[t](x^{k})=R_{i_{\ell(t)}^{t}}\cdots R_{i_{2}^{t}}R_{i_{1}^{t}}(x^{k}),\label{as1}
\end{equation}
and then calculating
\begin{equation}
x^{k+1}=\sum_{t\in\Theta_{k}}w_{k}(t)Z[t](x^{k}).\label{as2}
\end{equation}

This algorithmic scheme applies to $x^{k}$ successively the operators
$R_{i}:=U_{i}V_{i}$ whose indices belong to the string $t$. This
can be done in parallel for all strings and then the end-points of
all strings are convexly combined, with weights that may vary from
iteration to iteration, to form the next iterate $x^{k+1}$. This
is indeed an algorithm provided that the operators $\{R_{i}\}_{i=1}^{p}$
all have algorithmic implementations. In this framework we get a \textit{sequential
algorithm} by allowing a single string created by the index vector
$t=\Gamma$ and a \textit{simultaneous algorithm} by the choice\ of
$p$ different strings of length one each containing one element of
$\Gamma.$ Intermediate structures are possible by judicious choices
of strings and weights.

\section{Convergence\label{subsec:convergence}}

Next we prove the equivalence between Problem \ref{prob:moscfpp}
and a common fixed point problem which is not split, give a description
of Fix$(V_{j}),$ and state a property of $V_{j}$.

\begin{lemma} \label{Lem:a} Denote the solution set of Problem \ref{prob:moscfpp}
by $\Omega$ and assume that it is nonempty. Then, for $V_{j}$ as
in (\ref{eq:V_j}),

(i) $x^{\ast}\in\Omega$ if and only if $x^{\ast}$ solves the common
fixed point problem:
\begin{equation}
\text{Find }x^{\ast}\in\left(\cap_{i=1}^{p}\text{Fix}(U_{i})\right)\cap\left(\cap_{j=1}^{r}\text{Fix}(V_{j})\right),
\end{equation}
(ii) for all $j=1,2,\ldots,r$:
\begin{equation}
\text{Fix }(V_{j})=\{x\in\mathcal{H}\mid A_{j}x\in\text{Fix }(T_{j})\}=A_{j}^{-1}(\text{Fix }(T_{j})),
\end{equation}
{where $A_{j}^{-1}$ denotes here the inverse image
(pre-image) of $A_{j}$.} I.e., $A_{j}^{-1}:\mathcal{K}\rightarrow\mathcal{H}$
and for any $y\in\mathcal{K}$, $A_{j}^{-1}(y):=\{x\in\mathcal{H}\mid A_{j}x=y\}$;

(iii) if, in addition, all operators $T_{j}$ are cutters then all
$V_{j}$ are cutters (i.e., are 1-SQNE),

\noindent (iv) if $T_{j}$ is $\rho$-SQNE, $A_{j}\cap\text{Fix}T_{j}\neq\emptyset$
{(here we refer to $A_{j}$ as the image set of $A_{j}$)}
and satisfies the demi-closedness principle then $V_{j}$ also satisfies
the demi-closedness principle. \end{lemma}

\begin{proof} (i) We need to show only that
\begin{equation}
x^{\ast}\in\cap_{j=1}^{r}\text{Fix}(V_{j})\Leftrightarrow A_{j}x^{\ast}\in\text{Fix }(T_{j})\text{ for all }j=1,2,\ldots,r.
\end{equation}

Indeed, for any $j=1,2,\ldots,r,$
\begin{gather}
A_{j}x^{\ast}\in\text{Fix }(T_{j})\Leftrightarrow A_{j}x^{\ast}-T_{j}A_{j}x^{\ast}=0\nonumber \\
\Leftrightarrow{{}A_{j}^{\star}}(Id-T_{j})A_{j}x^{\ast}=A_{j}^{\star}0\text{ }\Leftrightarrow-\gamma_{j}A_{j}^{\star}(Id-T_{j})A_{j}x^{\ast}=0\nonumber \\
\Leftrightarrow x^{\ast}-\gamma_{j}A_{j}^{\star}(Id-T_{j})A_{j}x^{\ast}=x^{\ast}\Leftrightarrow x^{\ast}\in\text{Fix}(V_{j}).\label{eq:equivalence}
\end{gather}

(ii) Follows from (\ref{eq:equivalence}).

(iii) To show that $V_{j}$ is a cutter take $w\in\text{Fix}(V_{j})$,
$\gamma_{j}\in\left(0,\frac{1}{L_{j}}\right)$ and $\xi\in\mathcal{H}$.
\begin{align}
 & \frac{1}{\gamma_{j}}\left\langle w-V_{j}(\xi),\xi-V_{j}(\xi)\right\rangle \nonumber \\
 & =\left\langle w-\xi-\gamma_{j}A_{j}^{\star}(T_{j}-Id)A_{j}\xi,A_{j}^{\star}(Id-T_{j})A_{j}\xi\right\rangle \nonumber \\
 & =\left\langle w-\xi,A_{j}^{\star}(Id-T_{j})A_{j}\xi\right\rangle +\gamma_{j}\Vert A_{j}^{\star}(Id-T_{j})A_{j}\xi\Vert^{2}\nonumber \\
 & =\left\langle A_{j}w-A_{j}\xi,(Id-T_{j})A_{j}\xi\right\rangle +\gamma_{j}\Vert A_{j}^{\star}(Id-T_{j})A_{j}\xi\Vert^{2}\nonumber \\
 & =\left\langle A_{j}w-T_{j}(A_{j}\xi),(Id-T_{j})A_{j}\xi\right\rangle +\gamma_{j}\Vert A_{j}^{\star}(Id-T_{j})A_{j}\xi\Vert^{2}\nonumber \\
 & \ \ \ \ \ -\Vert(Id-T_{j})A_{j}\xi\Vert^{2}.
\end{align}
Since $T_{j}$ is a cutter and $A_{j}w\in\text{Fix }(T_{j})$, we
have
\begin{equation}
\left\langle A_{j}w-T_{j}(A_{j}\xi),(Id-T_{j})A_{j}\xi\right\rangle \leq0.
\end{equation}
Also,
\begin{equation}
\gamma_{j}\Vert A_{j}^{\star}(Id-T_{j})A_{j}\xi\Vert^{2}\leq\gamma_{j}\Vert A_{j}\Vert^{2}\Vert(Id-T_{j})A_{j}\xi\Vert^{2}\leq\Vert(Id-T_{j})A_{j}\xi\Vert^{2},
\end{equation}
for all $\gamma_{j}\in(0,1/L_{j})$. Using the above we get that
\begin{equation}
\left\langle w-V_{j}(\xi),\xi-V_{j}(\xi)\right\rangle \leq0.
\end{equation}
which proves that $V_{j}$ is a cutter.

(iv) Proved in \citet[Theorem 8(iv)]{cegielski14}. \end{proof}

The special case where in Problem \ref{prob:moscfpp} there is only
one operator $A:\mathcal{H}\rightarrow\mathcal{K}$ and (\ref{eq:prob1-line2})
{is replaced by}
\begin{equation}
\text{for all }1\leq j\leq r,\text{ }Ax^{\ast}\in\text{Fix }(T_{j})
\end{equation}
which amounts to $Ax^{\ast}\in\cap_{j=1}^{r}$Fix $(T_{j})$ was treated
in the literature \citep[see, e.g.,][]{cegielski15,cegielski14,wx11}.
The extensions to our more general case, necessitated by the application
to RTTP at hand, follow the patterns in those earlier papers. In our
convergence analysis we rely on the convergence result of \citet[Theorem 4.1]{sr16}
who, motivated by \citet[Algorithm 3.3]{cen-tom03}, invented and
investigated the ``modular string averaging (MSA) method''\ \citep[Procedure 1.1]{sr16}.

{For the convenience of the readers we quote next
in full details Procedure 1.1 and Theorem 4.1 of \citet{sr16}. We
adhere to the original notations of Reich and Zalas and later identify
them with the notations of our work. Let $U_{i}:\mathcal{H}\rightarrow\mathcal{H}$
be a finite family of quasi-nonexpansive mappings where $i\in I:=\{1,2,\cdots,M\}$
and define $U_{0}:=Id$. The problem under investigation is the common
fixed point problem of finding an $x\in C:=\cap_{i\in I}$Fix$(U_{i}).$
The algorithmic scheme is
\begin{equation}
x^{0}\in\mathcal{H},\;\;x^{k+1}=T_{k}x^{k}.\label{Alg:GenOpial}
\end{equation}
where the operator $T_{k}$ depends on a chosen subset of the input
operators ${U_{i}}.$}

{Reich and Zalas proposed Procedure 1.1 for constructing
operators $T_{k}$ (called ``modules'') is as follows. Fix $N\in\mathbb{N}$,
for $n=1,2,\ldots,N$; let $\varepsilon\in(0,1)$ be a fixed parameter;
define modules $V_{n}:=U_{-n}$ for all $n=-M,\ldots,0$. For $n=1,2,\ldots,N$
define modules $V_{n}$ by choosing one of the following cases:}

{(a) Relaxation: Fix a singleton $J_{n}=\{j_{n}\}\subseteq\{-M,\ldots,0\}$
and a relaxation $\alpha_{n}\in[\varepsilon,2-\varepsilon]$, and
set
\begin{equation}
V_{n}:=Id+\alpha_{n}\left(V_{j_{n}}-Id\right).
\end{equation}
}

{(b) Convex combination: Fix a nonempty subset $J_{n}\subseteq\{-M,\ldots,n-1\}$
and weights $\omega_{j,n}\in[\varepsilon,1-\varepsilon]$ satisfying
$\sum_{j\in J_{n}}\omega_{j,n}=1$, and set
\begin{equation}
V_{n}:=\sum_{j\in J_{n}}\omega_{j,n}V_{j}.
\end{equation}
}

{(c) Composition: Fix a ``string'' $J_{n}\subseteq\{-M,\ldots,n-1\}$
with length less than $M+n$ and set
\begin{equation}
V_{n}:=\Pi_{j\in J_{n}}V_{j}.
\end{equation}
}

{Using the above Modular String Averaging (MSA) procedure
of Reich and Zalas, by preforming $N_{k}$ steps with parameter $\varepsilon_{k}>0$,
$T_{k}$ is defined as the last module from the pool, that is, $T_{k}:=V_{N_{k}}^{k}$.
Such constructions of the operators $T_{k}$ lead to various combination
schemes such as: sequential, convex combination and composition. A
string averaging (SA) scheme that is relevant to our method here is
obtained by taking a convex combination of multiple compositions,
as in \citet[Equation (1.12)]{sr16}.}

{Reich and Zalas Theorem 4.1 is quoted next.}

\begin{theorem}\label{Th:sr_4.1} {Let $\left\{ x^{k}\right\} _{k=0}^{\infty}$
be a sequence generated by the iterative method}
\begin{equation}
{x^{0}\in\mathcal{H},\quad x^{k+1}=T_{k}(x^{k})}
\end{equation}

{and assume that:}

{(i) each operator $U_{i}$, $i\in I$ is a cutter;}

{(ii) $I\subseteq I_{k}\cup I_{k+1}\cup\cdots\cup I_{k+s-1}$,
for each $k=0,1,2,\ldots,$ and some $s\geq M-1$;}

{(iii) the sequence $\{N_{k}\}_{k=0}^{\infty}$ is
bounded.}

{If, for each $i\in I$, the operator $U_{i}$ satisfies
Opial's demi-closedness principle then the sequence $\left\{ x^{k}\right\} _{k=0}^{\infty}$
converges weakly to some point in $C$.}

{If, for each $i\in I$, the operator $U_{i}$ is
approximately shrinking and the family $\mathcal{C}:=\{\mathrm{Fix} \ U_{i}\mid i\in I\}$
is boundedly regular then the sequence $\left\{ x^{k}\right\} _{k=0}^{\infty}$
converges strongly to some point in $C$.}

\end{theorem}

{Our convergence theorem for the dynamic string-averaging
CQ-method now follows.}

\begin{theorem} Let $p\geq1$ be an integer and suppose that Problem
\ref{prob:moscfpp} with $r=p$ has a nonempty solution set $\Omega$.
Let $\{U_{i}\}_{i=1}^{p}$ and $\{T_{i}\}_{i=1}^{p}$ be cutters on
Hilbert spaces $\mathcal{H}$ and $\mathcal{K}$, respectively. Further
assume that $U_{i}-Id$ and $T_{i}-Id$ are demi-closed at zero for
all $i$. Then any sequence $\{x^{k}\}_{k=0}^{\infty},$ generated
by Algorithm \ref{alg:sa-cq} with $R_{i}:=U_{i}V_{i}$ for all $i,$
where $V_{i}$ are defined as in (\ref{eq:V_j}), converges weakly
to a point $x^{\ast}\in\Omega$. \end{theorem}

\begin{proof}

{First we identify the notations in our work with
those in \citet{sr16}.}

{(1) The operators $\{U_{i}\}_{i=1}^{M}$ of \citet[Theorem 4.1]{sr16}
are our $\{R_{i}\}_{i=1}^{p}$ where$R_{i}:=U_{i}V_{i}$ as described
in the beginning of Section \ref{Sec:Method} above.}\\
{{} (2) Our operators $\Psi_{\Theta_{k},w_{k}}$ (\ref{eq:algv})
are identified with the algorithmic operators $T_{k}$ of Equation
(1.12) in \citet{sr16}.}\\
{{} (3) Our operators $\{U_{i}\}_{i=1}^{p}$ and $\{T_{i}\}_{i=1}^{p}$
are assumed to be cutters, then so are also $\{V_{i}\}_{i=1}^{p}$,
by Lemma \ref{Lem:a}(iii). Hence, the composition operators $R_{i}:=U_{i}V_{i}$
are $\rho$-SQNE for all $i$ and, therefore, also our $\Psi_{\Theta_{k},w_{k}}$
are $\rho$-SQNE for all $k$.}\\
{{} (4) We assume that our $U_{i}-Id$ and $T_{i}-Id$
are demi-closed at zero for all $i$, therefore, by Lemma \ref{Lem:a}(iv),
$V_{i}-Id$ are also demi-closed at zero. So, our operators $R_{i}=U_{i}V_{i}$,
as composition of demi-closed operators, are demi-closed, see for
example \citep[Theorem 4.2]{cegielski15}. Our operators $R_{i}=U_{i}V_{i}$
are identified with $\{U_{i}\}_{i=1}^{M}$ of \citet{sr16}.}

Next we show that our dynamic string-averaging CQ-method fits into
the MSA \citep[Procedure 1.1]{sr16} and that the assumptions of \citet[Theorem 4.1]{sr16}
hold.

Since we identify our ${\Psi_{\Theta_{k},w_{k}}}$ from
(\ref{eq:algv}) with the right-hand side of Equation (1.12) of \citet{sr16}
(being careful with regard to the duplicity of symbols that represent
different things in that work and here), Algorithm \ref{alg:sa-cq}
can be represented by the iterative process of Equation (1.2), or
Equation (4.2), of \citet{sr16}.

Next we show the validity of the assumptions needed by \citet[Theorem 4.1]{sr16}.

Assumption (i) of \citet[Theorem 4.1]{sr16}: The operators $\{U_{i}\}_{i=1}^{M}$
of \citet[Theorem 4.1]{sr16} are our $R_{i}:=U_{i}V_{i}$. Although
our $R_{i}$ are not necessarily cutters, the arguments in the proof
of \citet[Theorem 4.1]{sr16} are based on the strongly quasi-nonexpansiveness
of the operators $T_{k}$ there (our ${\Psi_{\Theta_{k},w_{k}}}$)
and by Lemma \ref{Lem:a}(iii) above, our operators $\{V_{i}\}_{i=1}^{p}$
(defined in (\ref{eq:V_j})) are cutters and this together with the
assumption on our $\{U_{i}\}_{i=1}^{p}$ and $\{T_{i}\}_{i=1}^{p}$,
yields that the composition operators $R_{i}:=U_{i}V_{i}$ are $\rho$-SQNE
for all $i$ and, thus, so are also our ${\Psi}_{{\Theta_{k},w_{k}}}$.

Assumptions (ii)+(iii) of \citet[Theorem 4.1]{sr16}: Since the construction
of the operators ${\Psi}_{{\Theta_{k},w_{k}}}$
is based on $\mathcal{M}_{\ast}$ (\ref{eq:11516}) which mandates
a fit $\Theta$, it guarantees that every index $i\in\Gamma$ appears
in the construction of ${\Psi}_{{\Theta_{k},w_{k}}}$
for all $k>0$, thus, Assumption (ii) in \citet[Theorem 4.1]{sr16}
holds. Following the same reasoning, it is clear that the number of
steps $N_{k}$, defined in the MSA \citep[Procedure 1.1]{sr16}, is
bounded.

The weak convergence part of the proof of \citet[Theorem 4.1]{sr16}
requires that all (their) $\{U_{i}\}_{i=1}^{M}$ satisfy Opial's demi-closedness
principle (i.e., that $U_{i}-Id$ are demi-closed at zero). In our
case, we assume that $U_{i}-Id$ and $T_{i}-Id$ are demi-closed at
zero for all $i$. By Lemma \ref{Lem:a}(iv) above $V_{i}-Id$ are
also demi-closed at zero. So, we identify $\{U_{i}\}_{i=1}^{M}$ of
\citet{sr16} with our $U_{i}$s and $V_{i}$s and construct first
the operators $R_{i}=U_{i}V_{i}$, and then use them as the building
bricks of the algorithmic operators $\Psi_{\Theta_{k},w_{k}}$.

{Observe that in our proposed dynamic string-averaging
scheme the weights are chosen, in every iteration $k,$ so that $(\Theta_{k},w_{k})\in\mathcal{M}_{\ast}$
(see the iterative step of Algorithm \ref{alg:sa-cq}). This requires,
according to (\ref{eq:11516}), that $w(t)\geq\Delta\text{ for all }t\in\Theta,$
where $\Delta\in(0,1/p)$ is a fixed positive number. Therefore, for
any $t$ it must hold that $\sum_{k=0}^{\infty}w_{k}(t)=\infty$,
meaning that we ``visit'' every operator infinitely many times.
This fully coincides with the assumption in (Reich and Zalas, 2016)
that $w_{k}(i)\in[\varepsilon,1-\varepsilon]$ for some $\varepsilon>0$
which implies that $\sum_{k=0}^{\infty}w_{k}(i)=\infty$ for all $i$,
in their notation.}

Thus, the desired result is obtained. \end{proof}

\begin{remark} (i) If one assumes that the $T_{j}$ operators are
firmly nonexpansive, then similar arguments as in the proof of \citet[Theorem 3.1]{moudafi11}
show that the $V_{j}$ operators are also averaged and then \citet[Theorem 4.1]{sr16}
can be adjusted to hold for averaged operators.

(ii) It is possible to propose inexact versions of Algorithm \ref{alg:sa-cq}
following \citet[Theorem 4.5]{sr16} and Combettes' ``\textit{almost
cyclic sequential algorithm} (ACA)'' \citep[Algorithm 6.1]{combettes01}.

(iii) Our work can be extended to cover also underrelaxed operators,
i.e., by defining $R_{i}:=(U_{i})_{\lambda}(V_{i})_{\delta}$ for
$\lambda,\delta\in\lbrack0,1].$ This is allowed due the fact that
if an operator is firmly quasi-nonexpansive, then so is its relaxation.

(iv) \citet[Theorem 4.1]{sr16} also includes a strong convergence
part under some additional assumptions on their operators $\{U_{i}\}_{i=1}^{M}$.
It is possible to adjust this theorem for our case as well.

{(v) We proposed here a general scheme that allows
dynamic string averaging; the closest CQ variant appears in the work
of \citet[Theorem 3.1]{wx11} where only sequential, cyclically controlled,
iterations are allowed.}

{(vi) For the case of a two-set non-convex feasibility
problem, \citet[Theorem 5.3]{abs13} propose a CQ variant but without
a relaxation and if more than two non-convex sets are allowed, then
a fully simultaneous method is obtained.}\end{remark}

\section{Percentage violation constraints (PVCs) arising in radiation therapy
treatment planning\label{sect:PVC}}

\subsection{Transforming problems with a PVC\label{subsec:transforming}}

Given $p$ closed convex subsets $\varOmega_{1},\varOmega_{2},\cdots,\varOmega_{p}\subseteq\mathbb{R}^{n}$
of the $n$-dimensional Euclidean space $\mathbb{R}^{n}$, expressed
as level sets
\begin{equation}
\varOmega_{j}=\left\{ x\in\mathbb{R}^{n}\mid f_{j}(x)\leq v_{j}\right\} ,\text{ for all }j\in J:=\{1,2,\ldots,p\},\label{eq:cfp}
\end{equation}
where $f_{j}:\mathbb{R}^{n}\rightarrow\mathbb{R}$ are convex functions
and $v_{j}$ are some given real numbers, the convex feasibility problem
(CFP) is to find a point $x^{\ast}\in\varOmega:=\cap_{j\in J}\varOmega_{j}.$
If $\varOmega=\emptyset$ {then the CFP is said to
be inconsistent}.

\begin{problem} \label{prob:cfp-pvc}\textbf{Convex feasibility problem
(CFP) with a percentage-violation constraint (PVC) (CFP+PVC).} Consider
$p$ closed convex subsets $\varOmega_{1},\varOmega_{2},\cdots,\varOmega_{p}\subseteq\mathbb{R}^{n}$
of the $n$-dimensional Euclidean space $\mathbb{R}^{n}$, expressed
as level sets {according to (\ref{eq:cfp})}. Let
$0\leq\alpha\leq1$ and $0<\beta<1$ be two given real numbers. The
CFP+PVC is:

Find an $x^{\ast}\in\mathbb{R}^{n}$ such that $x^{\ast}\in\cap_{j=1}^{p}\varOmega_{j}$
and in\ up\ to\ a fraction\ $\alpha$ (i.e., 100$\alpha$\%) of\ the\ total
number of inequalities in (\ref{eq:cfp}) the\ bounds $v_{j}$\ may\ be\ potentially\ violated
by up\ to\ a fraction\ $\beta$ (i.e., 100$\beta$\%) of their
values.\ \end{problem}

A PVC is an integer constraint by its nature. It changes the CFP (\ref{eq:cfp})
to which it is attached from being a continuous feasibility problem
into becoming a mixed integer feasibility problem. Denoting the inner
product of two vectors in $\mathbb{R}^{n}$ by $\left\langle a,b\right\rangle :=\sum_{i=1}^{n}{a_{i}b_{i}}$,
the linear feasibility problem (LFP) with PVC (LFP+PVC) is the following
special case of Problem \ref{prob:cfp-pvc}.

\begin{problem} \label{prob:lfp-pvc}\textbf{Linear feasibility problem
(LFP) with a percentage-violation constraint (PVC) (LFP+PVC).} {This
is the same as Problem \ref{prob:cfp-pvc} with $f_{j}$, for $j=1,2,\ldots,p,$
in (\ref{eq:cfp}) being linear functions, meaning that the sets $\varOmega_{j}$
are half-spaces:
\begin{equation}
\varOmega_{j}=\left\{ x\in\mathbb{R}^{n}\mid\left\langle a^{j},x\right\rangle \leq b_{j}\right\} ,\text{ for all }j\in J,\label{eq:lfp}
\end{equation}
for a set of given vectors $a^{j}\in\mathbb{R}^{n}$ and $b_{j}$
some given real numbers.} \end{problem}

Our tool to ``translate''\ the mixed integer LFP+PVC into a ``continuous''\ one
is the notion of sparsity norm, called elsewhere the zero-norm, of
a vector $x\in\mathbb{R}^{n}$ which counts the number of nonzero
entries of $x,$ that is,
\begin{equation}
\Vert x\Vert_{0}:=\left\vert \{x_{i}\mid x_{i}\neq0\}\right\vert ,\label{eq:L0norm}
\end{equation}
where $\mid\cdot\mid$ denotes the cardinality, i.e., the number of
elements, of a set. This notion has been recently used for various
purposes in compressed sensing, machine learning and more. The rectifier
(or ``positive ramp operation'') on a vector $x\in\mathbb{R}^{n}$
means that, for all $i=1,2,\ldots,n,$
\begin{equation}
(x_{+})_{i}:=\max(0,x_{i})=\left\{ \begin{array}{l}
x_{i},\text{ if }x_{i}>0,\\
0,\text{ if }x_{i}\leq0.
\end{array}\right..
\end{equation}
Obviously, $x_{+}$ is always a component-wise nonnegative vector.
Hence, $\Vert x_{+}\Vert_{0}$ counts the number of positive entries
of $x$ and is defined by
\begin{equation}
\Vert x_{+}\Vert_{0}:=|\{x_{i}\mid x_{i}>0\}|.
\end{equation}
We translate the LFP+PVC to the following.

\begin{problem} \label{prob:lfp-pvc-upper}\textbf{Translated problem
of LFP+PVC (for LFP with upper bounds). }For the data of Problem \ref{prob:lfp-pvc},
{let $A\in\mathbb{R}^{p\times n}$ be the matrix whose
columns are formed by the vectors $a^{j}$ and let $b\in\mathbb{R}^{p}$
be the column vector consisting of the values $b_{j}$, for all $j\in J$.}
The translated problem of LFP+PVC (for LFP with upper bounds)\textbf{
}is:
\begin{gather}
\text{Find an }x^{\ast}\in\mathbb{R}^{n}\text{ such that }\left\langle a^{j},x^{\ast}\right\rangle \leq(1+\beta)b_{j},\text{}\label{eq:prob-trans1}\\
\text{for all }j\in J,\text{ and }\Vert(Ax^{\ast}-b)_{+}\Vert_{0}\leq\alpha p.
\end{gather}

\end{problem}

The number of the violations in (\ref{eq:prob-trans1}) is $\Vert(Ax^{\ast}-b)_{+}\Vert_{0}$
and $\Vert(Ax^{\ast}-b)_{+}\Vert_{0}\leq\alpha p$ guarantees that
the number of violations of up to $\beta$ in the original row inequalities
remains at bay as demanded. This is a split feasibility problem between
the space $\mathbb{R}^{n}$ and the space {$\mathbb{R}^{p}$}
with the matrix $A$ mapping the first to the latter. The constraints
in $\mathbb{R}^{n}$ are linear (thus convex) but the constraint
\begin{equation}
x^{\ast}\in S:=\{y\in{\mathbb{R}^{p}}\mid\Vert(y-b)_{+}\Vert_{0}\leq\alpha p\}
\end{equation}
is not convex. This makes Problem \ref{prob:lfp-pvc-upper} similar
in structure {to, but not identical with, Problem
\ref{prob:split-feas}}.

Similarly, if the linear inequalities in Problem \ref{prob:lfp-pvc-upper}
are in an opposite direction, i.e., of the form $c_{j}\leq\left\langle a^{j},x\right\rangle ,$
for all $j\in J,$ then the translated problem of LFP+PVC will be
as follows.

\begin{problem} \label{prob:lfp-pvc-lower}\textbf{Translated problem
of LFP+PVC (for LFP with lower bounds). }For the data of Problem \ref{prob:lfp-pvc},
{let $A\in\mathbb{R}^{p\times n}$ be the matrix whose
columns are formed by the vectors $a^{j}$ and let $c\in\mathbb{R}^{p}$
be the column vector consisting of the values $c_{j}$, for all $j\in J$.}
The translated problem of LFP+PVC (for LFP with lower bounds) is:
\begin{gather}
\text{Find an }x^{\ast}\in\mathbb{R}^{n}\text{ such that }(1-\beta)c_{j}\leq\left\langle a^{j},x^{\ast}\right\rangle ,\text{}\label{eq:prob2-trans1}\\
\text{for all }j\in J,\text{ and }\Vert(c-Ax^{\ast})_{+}\Vert_{0}\leq\alpha p.
\end{gather}
\end{problem}

This is also a split feasibility problem between the space $\mathbb{R}^{n}$
and the space {$\mathbb{R}^{p}$} with the matrix
$A$ mapping the first to the latter. The constraints in $\mathbb{R}^{n}$
are linear (thus convex) but the constraint
\begin{equation}
x^{\ast}\in T:=\{y\in{\mathbb{R}^{p}}\mid\Vert(c-y)_{+}\Vert_{0}\leq\alpha{p}\}
\end{equation}
is again not convex.

\subsection{Translated block LFP+PVC}

\label{subsec:tarns-block}

Consider an $m\times n$ matrix $A$ divided into blocks $A_{\ell},$
for $\ell=1,2,\ldots,\Gamma,$ with each block forming an $m_{\ell}\times n$
matrix and $\sum_{_{\ell=1}}^{\Gamma}m_{\ell}=m.$ Further, the blocks
are assumed to give rise to block-wise LFPs of the two kinds; those
with upper bounds, say for $\ell=1,2,\ldots,p,$ and those with lower
bounds, say for $\ell=p+1,p+2,\ldots,p+r$. PVCs are imposed on each
block separately with parameters $\alpha_{\ell}$ and $\beta_{\ell}$,
respectively, for all $\ell=1,2,\ldots,\Gamma.$ The original block-LFP
prior to imposing the PVCs is:
\begin{equation}
\begin{array}{cc}
A_{\ell}x\leq b^{\ell}, & \text{for all }\ell=1,2,\ldots,p,\\
c^{\ell}\leq A_{\ell}x, & \text{for all }\ell=p+1,p+2,\ldots,p+r.
\end{array}
\end{equation}
{Such constraints will be termed ``hard dose constraints" (HDCs).} After imposing the PVCs and translating the systems according to the
principles of Problems \ref{prob:lfp-pvc-upper} and \ref{prob:lfp-pvc-lower}
we obtain the translated problem of LFP+PVC for blocks.

\begin{problem} \label{prob:trans-blocks}\textbf{Translated problem
of LFP+PVC for blocks. }Find an $x^{\ast}\in\mathbb{R}^{n}$ such
that
\begin{equation}
\begin{array}{cc}
A_{\ell}x^{\ast}\leq(1+\beta_{\ell})b^{\ell}, & \text{for all }\ell=1,2,\ldots,p,\\
(1-\beta_{\ell})c^{\ell}\leq A_{\ell}x^{\ast}, & \text{for all }\ell=p+1,p+2,\ldots,p+r,\\
\Vert(A_{\ell}x^{\ast}-b^{\ell})_{+}\Vert_{0}\leq\alpha_{\ell}m_{\ell}, & \text{for all }\ell=1,2,\ldots,p,\\
\Vert(c^{\ell}-A_{\ell}x^{\ast})_{+}\Vert_{0}\leq\alpha_{\ell}m_{\ell}, & \text{for all }\ell=p+1,p+2,\ldots,p+r.
\end{array}
\end{equation}

\end{problem}

This is a split feasibility problem between the space $\mathbb{R}^{n}$
and the space $\mathbb{R}^{m}$ but with a structure similar to Problem
\ref{prob:massad} where, for $\ell=1,2,\ldots,\Gamma,$ each $A_{\ell}$
maps $\mathbb{R}^{n}$ to $\mathbb{R}^{m_{\ell}}.$ Again, it is not
identical with Problem \ref{prob:massad} because here the constraints
in $\mathbb{R}^{m_{\ell}}$, for $\ell=1,2,\ldots,\Gamma,$ are not
convex. . Although Problem \ref{prob:trans-blocks} defines an upper
PVC on exactly $p$ blocks and a lower PVC on exactly $r$ blocks,
we can, without loss of generality, choose to define PVCs only on
a subset of these blocks. For blocks without a PVC, the problem reverts
to a standard LFP. \\

\section{Application to radiation therapy treatment planning\label{sec:application}}

The process of planning a radiotherapy treatment plan involves a physician
providing dose prescriptions which geometrically constrain the distribution
of dose deposited in the patient. Choosing the appropriate nonnegative
weights of many individual beamlet dose kernels to achieve these prescriptions
as best as possible is posed as a split inverse problem. We focus,
for our purposes, on constraining the problem with upper and lower
dose bounds, and dose volume constraints (DVCs), which we more generally
refer to as PVCs in this work. DVCs allow dose levels in a specified
proportion of a structure to fall short of, or exceed, their prescriptions
by a specified amount. They largely serve to allow more flexibility
in the solution space.

Problem \ref{prob:trans-blocks} describes the split feasibility problem
as it applies in the context of radiation therapy treatment planning.
Each block represents a defined geometrical structure in the patient,
which is classified either as an \textit{avoidance structure} or a
\textit{target volume}. An example of an avoidance {structure}
is an organ at risk (OAR), in which one wishes to deposit minimal
dose. An example of a target {structure} is the planning
target volume (PTV), to which a sufficient dose is prescribed to destroy
the tumoural tissue. If there are $p$ avoidance structures, any number
of blocks in $\{1,2,\ldots,p\}$ can have lower PVCs applied. Similarly,
if there are $r$ target volumes then any number of blocks in $\{p+1,p+2,\ldots,p+r\}$
can have an upper PVC applied.

This problem can be formulated as the MOSCFPP described in Problem
\ref{prob:moscfpp} as follows. For the data of Problem \ref{prob:trans-blocks},
define ${\bar{\Gamma}}\subseteq\{1,2,\ldots,p+r\}$ and
for all $i=1,2,\ldots,m_{\ell}$, let
\begin{equation}
C_{\ell}^{i}:=\{x\in\mathbb{R}_{+}^{n}\mid\langle a_{\ell}^{i},x\rangle\leq(1+\beta_{\ell})b_{i}^{\ell}\},
\end{equation}
for all $\ell\in\{1,2,\ldots,p\}$ where $\mathbb{R}_{+}^{n}$ is
the nonnegative orthant, and
\begin{equation}
C_{\ell}^{i}:=\{x\in\mathbb{R}_{+}^{n}\mid(1-\beta_{\ell})c_{i}^{\ell}\leq\langle a_{\ell}^{i},x\rangle\},
\end{equation}
for all $\ell\in\{p+1,p+2,\ldots,p+r\}$. Additionally, let
\begin{equation}
Q_{\ell}:=\{{A_{\ell}x=v\in\mathbb{R}^{m_{\ell}}}\mid\Vert({v}-b^{\ell})_{+}\Vert_{0}\leq\alpha_{\ell}m_{\ell}\},
\end{equation}
for all $\ell\in\{1,2,\ldots,p\}\cap{\bar{\Gamma}}$ and
\begin{equation}
Q_{\ell}:=\{{A_{\ell}x=v\in\mathbb{R}^{m_{\ell}}}\mid\Vert(c^{\ell}-{v})_{+}\Vert_{0}\leq\alpha_{\ell}m_{\ell}\}
\end{equation}
for all $\ell\in\{p+1,p+2,\ldots,p+r\}\cap{\bar{\Gamma}}$. {The above $A_\ell$ are blocks of the original matrix $A$ and we denote by $A_\ell x = v$ the image of the vector $x$ under $A_\ell$.}

\begin{problem} \label{prob:moscfpp_rttp}\textbf{Translated problem
of MOSCFPP for RTTP.}

Let the operators $P_{C_{\ell}^{i}}:\mathbb{R}^{n}\rightarrow\mathbb{R}^{n}$
be orthogonal projections onto $C_{\ell}^{i}$ for all $\ell\in\{1,2,\ldots,p+r\}$
and $i\in\{1,2,\ldots,m_{\ell}\}$, and let $P_{Q_{\ell}}:\mathbb{R}^{m_{\ell}}\rightarrow\mathbb{R}^{m_{\ell}}$
be orthogonal projections onto $Q_{\ell},$ for all $\ell\in\varGamma$.
The translated MOSCFPP for RTTP is:

\begin{align}
\text{Find } & \text{an }x^{\ast}\in\mathbb{R}_{+}^{n}\text{ such that }x^{\ast}\in\bigcap_{\ell=1}^{p+r}\bigcap_{i=1}^{m_{\ell}}\text{\ensuremath{\mathrm{Fix}}}(P_{C_{\ell}^{i}})\text{ and, }\nonumber \\
 & \text{for all }\ell\in\varGamma,\text{ }A_{\ell}x^{\ast}\in\text{\ensuremath{\mathrm{Fix}}}(P_{Q_{\ell}}).
\end{align}

\end{problem}

We seek a solution to Problem \ref{prob:moscfpp_rttp} using our dynamic
string-averaging CQ-method, described in Algorithm \ref{alg:sa-cq}.
We define, for all $\ell\in\Gamma$,
\begin{equation}
V_{\ell}:=Id-\gamma_{\ell}A_{\ell}^{T}(Id-P_{Q_{\ell}})A_{\ell},\label{eq:V_l_PQ}
\end{equation}

where $\gamma_{\ell}\in\left(0,\frac{1}{L_{\ell}}\right)$, $L_{\ell}=\Vert A_{\ell}\Vert^{2}$
{and $A_{\ell}^{T}$ is the transpose of $A_{\ell}$}.

\begin{remark} \label{rem:adaptiveStep} In practical use relaxation
parameters play an important role:

(i) Each projection operator $P_{C_{\ell}^{i}}:\mathbb{R}^{n}\rightarrow\mathbb{R}^{n}$
may be relaxed with a parameter $\lambda_{\ell}\in(0,2)$ defined
on the block $\ell\in\{1,2,\ldots,p+r\}$.

(ii) The relaxation parameters $\lambda_{\ell},$ as defined in (i),
and $\gamma_{\ell}$, as given in (\ref{eq:V_l_PQ}), are permitted
to take any value within their bounds on any iterative step of Algorithm
\ref{alg:sa-cq}. That is, they may depend on (vary with) the iteration
index $k$ and, therefore, be labeled $\lambda_{\ell}^{k}$ and $\gamma_{\ell}^{k}$.

(iii) {The sets $Q_{\ell}$ are nonconvex and if for
a given $\alpha_{\ell}m_{\ell}$ it is nonempty, then it is also closed
and then projection onto $Q_{\ell}$ exists, is not necessarily} unique,
but can be calculated explicitly, see, e.g., \citep[Eq. (24)]{pzcbcs17}.
For properties regarding similar sets see, e.g., \citep[Subsection 6.8.3]{Beck17}.
{A recent work of \citet{hln14}
includes an investigation of these questions, see Section III there.
Answers about the sets $Q_{\ell}$ and projections onto them in the
specific setting related to the radiation therapy treatment planning
problem considered here are not yet available.} \end{remark}

Tracking the percentage of elements in the current iteration of dose
vectors $A_{\ell}x^{k}$ that are violating their constraints enables
one to impose an adaptive version of Algorithm \ref{alg:sa-cq} using
the comments in {Remark \ref{rem:adaptiveStep}}.
If, for example, one block has more PVC violations than LFP (dose
limit constraints) violations then one could choose to alter the relaxation
parameters at the next iteration, $\lambda_{\ell}^{k+1}$ and $\gamma_{\ell}^{k+1}$,
in order to place less emphasis on the projections onto $C_{\ell}^{i}$.

\section{Numerical implementation}\label{sec:numerical}

\subsection{Operator definitions}

In Problem \ref{prob:moscfpp_rttp} we introduced the orthogonal projection operators $P_{C_{\ell}^{i}}$, which act in the space of the pencil beam intensity vector $x$, and $P_{Q_{\ell}}$, which act in the space of the dose vector $A_\ell x$. Here we provide explicit formulae, as examples, for calculating these projections in practice. Given an arbitrary vector $z\in \mathbb{R}^n$ and some $\ell\in\{1,2,\ldots,p+r\}$ and $i\in\{1,2,\ldots,m_{\ell}\}$, if it is the case that $z$ is not in $C_\ell^i$ then it must be projected onto the nearest hyperplane which defines the boundary of $C_\ell^i$. Otherwise, no action is taken. If block $\ell$ represents an avoidance structure ($\ell \in \{ 1,2,\hdots,p \}$) then the projection can be calculated by
\begin{equation}
    P_{C_\ell^i}(z) = \begin{dcases}
        z, & \langle a_\ell^i,z \rangle \leq (1+\beta_\ell)b_i^\ell,
        \\
        z + \lambda_\ell\frac{(1+\beta_\ell)b_i^\ell-\langle a_\ell^i,z \rangle}{\langle a_\ell^i,a_\ell^i \rangle}a_\ell^i, & \langle a_\ell^i,z \rangle > (1+\beta_\ell)b_i^\ell,
    \end{dcases}\label{eqn:projC_avoid}
\end{equation}
where $\lambda_\ell \in (0,2)$ is a user-selected relaxation parameter. Alternatively, if $\ell$ represents a target structure ($\ell \in \{ p+1,p+2,\hdots,p+r \}$) then the projection can be similarly calculated using
\begin{equation}
    P_{C_\ell^i}(z) = \begin{dcases}
        z, & \langle a_\ell^i,z \rangle \geq (1-\beta_\ell)c_i^\ell,
        \\
        z + \lambda_\ell\frac{(1-\beta_\ell)c_i^\ell-\langle a_\ell^i,z \rangle}{\langle a_\ell^i,a_\ell^i \rangle}a_\ell^i, & \langle a_\ell^i,z \rangle < (1-\beta_\ell)c_i^\ell.
    \end{dcases}\label{eqn:projC_target}
\end{equation}
{Note that, since in the above $\lambda_\ell \in (0,2)$ are used, then $P_{C_\ell^i}(z)$ are \emph{relaxed} projections.}

It is of interest to note that in clinical practice a structure may well have both an upper bound and a lower bound placed on the permitted dose. Such cases can be handled by simply defining two blocks for the same structure, one as an avoidance block, to which (\ref{eqn:projC_avoid}) applies, and one as a target block, to which (\ref{eqn:projC_target}) applies.

Projection of the dose vector onto $Q_\ell$ follows a slightly more elaborate procedure. We first define a helper set,
\begin{equation}
\overline{Q}_\ell := \{ y\in \mathbb{R}^{m_\ell} \mid \Vert y_+ \Vert_0 \leq \alpha_\ell m_\ell \},
\end{equation}
and describe the projection onto the set, $P_{\overline{Q}_\ell}$, by the following rules: for an arbitrary vector $y\in\mathbb{R}^{m_\ell}$, first count the number of positive entries, $\Vert y_+ \Vert_0$. If $\Vert y_+ \Vert_0\leq \alpha_\ell m_\ell$ then the vector is in $\overline{Q}_\ell$ and no action is needed; $P_{\overline{Q}_\ell} = Id$, the identity operator. However, if $\Vert y_+ \Vert_0 > \alpha_\ell m_\ell$ then $P_{\overline{Q}_\ell}$ replaces the $\lfloor  (\Vert y_+ \Vert_0 - \alpha_\ell m_\ell ) \rfloor$ smallest positive components of $y$ with zeros and leaves the others unchanged. We can now define $P_{Q_\ell}$ in terms of a projection onto the helper set. Given $v\in\mathbb{R}^{m_\ell}$,
\begin{equation}
P_{Q_\ell}(v) =
\begin{dcases}
P_{\overline{Q}_\ell}(v-b^\ell)+b^\ell, & \ell \in \{ 1,2,\hdots,p \} \cap \overline{\Gamma},
\\
-P_{\overline{Q}_\ell}(c^\ell-v)+c^\ell, & \ell \in \{ p+1,p+2,\hdots,p+r \} \cap \overline{\Gamma}.
\end{dcases}
\end{equation}

{Since the sets $\overline Q_\ell$ are non-convex, the projection is not unique and so it might happen that only one of the possible values has to be chosen. The reader is referred to related results by \citet[Proposition 3.1]{{lz12}}, \citet[Equation (20)]{hln14} and \citet[Page 54]{Schaad10}.}

\subsection{Inverse planning algorithm}

We provide here a practical example of how Algorithm \ref{alg:sa-cq} may be implemented for inverse planning in radiation therapy treatment planning. In this example we initialize each of the beamlet weights to unit intensity, $x^0=(1,1,\hdots,1)^T$, before running through multiple cycles of an iterative scheme that is equivalent to a fully sequential Algorithm \ref{alg:sa-cq} with unit weights, $w_k=1$ for all $k$, in (\ref{as2}). The pseudo-code of this procedure is detailed in Algorithm \ref{alg:sa-cq-rttp}. The two ``for" loop control cycles therein imply that the blocks, $\ell$, may be chosen in any order, without replacement, and so may the voxels, $i$, within each block. Within each cycle, a nonnegativity constraint is enforced after all possible projections have been applied. This sets any unphysical negative entries in the beamlet intensity vector, $x$, to zero. In this example a preset number of cycles are performed before stopping and accepting the final solution. However, one may easily replace this by a tolerance-based stopping criterion.

\setcounter{theorem}{1}
\begin{algo}\label{alg:sa-cq-rttp}\textbf{The dynamic string-averaging CQ-method: A pseudo-code \\ example for RTTP}\\
\begin{algorithm}[H]
\small
\SetAlgoLined
 \textbf{Initialization:} $x^0 = (1,1,\hdots,1)^T$, cycle number $k=1$, choose max cycles $N_\mathrm{cycles}$\;
 \While{$k<N_\mathrm{cycles}$}{
 \For{$\ell \in \{ 1,2,\hdots,p,p+1,\hdots,p+r \}$}{
 	\If{$\ell \in \overline{\Gamma}$}{
  		Choose some $0<\gamma_\ell<2/\Vert A_\ell \Vert^2$\;
  		$x^k\leftarrow x^k-\gamma_\ell A_\ell^T(A_\ell x^k-P_{Q_\ell}(A_\ell x^k))$\;
  	}
  	\For{$i \in \{ 1,2,\hdots,m_\ell \}$}{
  		$x^k\leftarrow P_{C_\ell^i}(x^k)$\;
  	}
  	$x^{k+1}\leftarrow x^k_+$ (enforce nonnegativity constraint)\;
  	$k\leftarrow k+1$\;
 }
 }
\end{algorithm}
\end{algo}

\subsection{Numerical example}

A two-dimensional pseudo-dose grid was created using \emph{MATLAB}, version R2019a (The MathWorks, Inc., 2020). The grid is made of a matrix of dimensions $512 \times 512$ representing 262,144 pixels which altogether comprise an area of dosimetric interest. In a clinical treatment plan this would be the entire patient geometry and the pixels would be replaced by a large number of three-dimensional voxels. Without loss of generality, we assume two spatial dimensions for simplicity. In order to achieve a basic emulation of dose deposited by multiple beamlets, 1,156 Gaussian pseudo-dose kernels were uniformly distributed across the grid. Each kernel had a standard deviation of 20 pixels and an amplitude such that their sum produced a homogeneous intensity map, with a mean value of 50 units. Figure \ref{fig:dosemaps}(a) shows a visualization of the intensity (pseudo-dose) matrix due to a single Gaussian kernel, with each dotted grid point representing the centre of one of the 1,156 kernels. Figure \ref{fig:dosemaps}(b) shows the sum of all contributions. Note that each kernel contributes equally to the sum at this stage, prior to the inverse planning procedure. From this point on, for the proper RTTP context, we shall assume that pixel values directly correspond to ``dose".

\setcounter{table}{0}
\begin{table}[b]
\caption{{\footnotesize Prescription chosen for the two-dimensional numerical example. Pseudo-dose units are arbitrary. $D_{V\%}$ represents the dose that is received by exactly V\% of the structure. $D_\mathrm{max}$ and $D_\mathrm{min}$ represent the maximum and minimum dose constraints, respectively.}}\label{tab:numerical}
\begin{tabular}{lrr}
Structure   & HDCs & DVCs \\ \hline
Avoidance A & $D_\mathrm{max}=25$    & $D_\mathrm{10\%}\leq 20$    \\ \hline
Avoidance B & $D_\mathrm{max}=40$ 	 & $D_\mathrm{25\%}\leq 30$    \\ \hline
Target      & $D_\mathrm{min}=60$    & $D_\mathrm{90\%}\geq 65$    \\
            & $D_\mathrm{max}=70$    &    \\ \hline
\end{tabular}
\end{table}

\begin{figure}[b!]
\makebox[\textwidth][c]{\includegraphics[width=0.95\columnwidth]{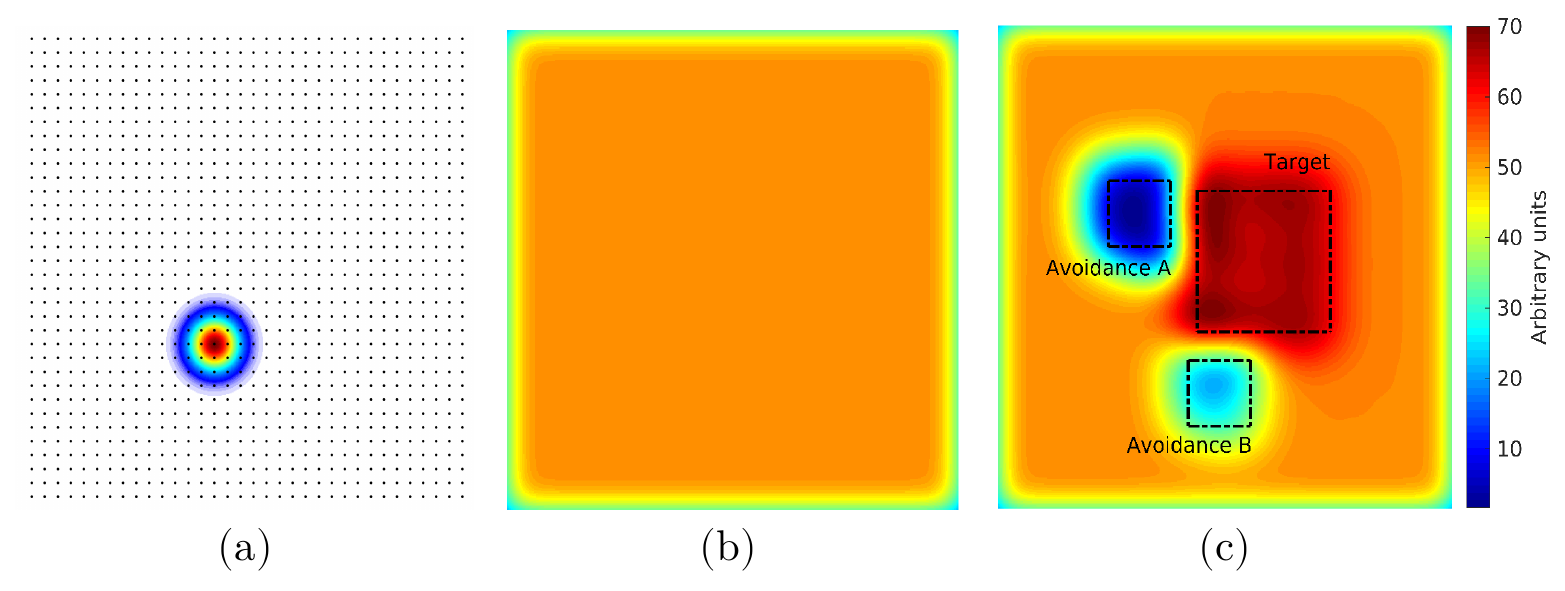}}
\caption{{\footnotesize (a) A single Gaussian pseudo-dose kernel contribution shown at one grid point. (b) Homogeneous pseudo-dose of 50 units formed by superimposing all 1,156 Gaussian contributions. (c) Optimized pseudo-dose map showing the structures for which the prescription in Table \ref{tab:numerical} was applied.}}\label{fig:dosemaps}
\end{figure}

We have thus far introduced 1,156 different matrices of dimensions $512 \times 512$. In order to form a \emph{dose-influence} matrix, $A$, for use in inverse planning, each matrix is collapsed to a single column vector with 262,144 entries, making sure to keep track of which indices corresponded to which spatial positions in the dose grid. The matrix $A$ is formed by all column vectors and therefore has $262,144$ rows and $1156$ columns.

A prescription composed of four hard dose constraints (HDCs), for minimum and maximum dose bounds, and three DVCs, shown in Table \ref{tab:numerical}, was applied to three arbitrarily defined disjoint square regions. DVCs in Table \ref{tab:numerical} are written in the standard notation, $D_{V\%}$, which is the dose that is received by exactly $V$\% of the structure. In the framework of this paper, an upper DVC on block $\ell$ is equivalent to writing $D_{100\alpha_\ell\%}\leq b^\ell$ and a lower DVC is equivalent to $D_{100\alpha_\ell\%}\geq c^\ell$. $D_\mathrm{max}$ and $D_\mathrm{min}$ represent the maximum and minimum dose constraints, respectively. The three defined square regions can be seen overlaying the dose solution in Figure \ref{fig:dosemaps}(c). These consist of two avoidance regions, ``Avoidance A" and ``Avoidance B", and one target region, ``Target". The column indices of the matrix $A$ corresponding to pixels inside the boundary of these regions can be used to form submatrices, $A_1$, $A_2$ and $A_3$, respectively.

\begin{figure}[t!]
\makebox[\textwidth][c]{\includegraphics[width=0.52\columnwidth]{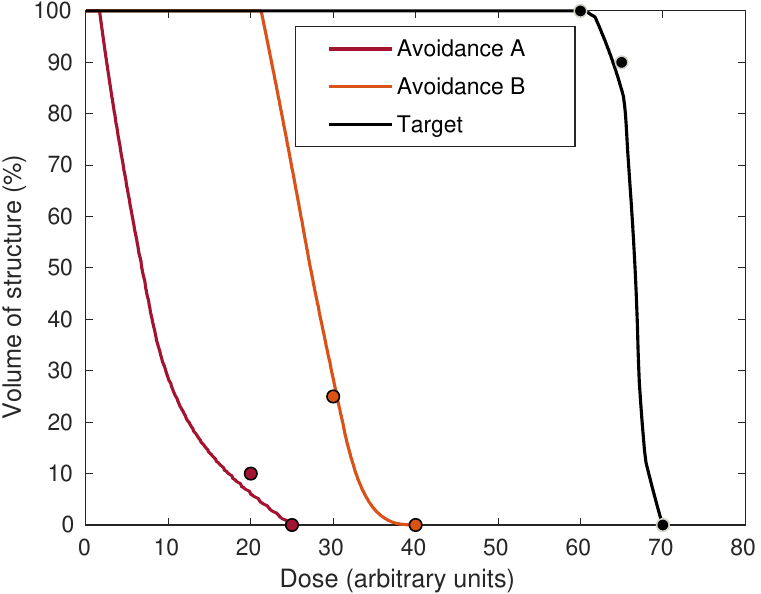}}
\caption{{\footnotesize Cumulative dose-volume histogram (DVH) showing the percentage of each structure that has received a certain dose. HDC and DVC prescriptions are shown as filled circles.}}\label{fig:DVH}
\end{figure}

We now have a framework in which Algorithm \ref{alg:sa-cq-rttp} can be applied. We have $A_\ell$ for $\ell \in \{ 1,2,3 \}$ with $p=2$ and $r=1$, and we have $x^0=(1,1,\hdots,1)^T$ with 1,156 entries. In this particular case, both a lower and upper bound on the dose have been prescribed for the ``Target" structure. Therefore, we will actually use $A_\ell$ for $\ell \in \{ 1,2,3,4 \}$, where $A_4=A_3$ and $\ell=3$ corresponds to the minimum dose constraint while $\ell=4$ corresponds to the maximum dose constraint.

Algorithm \ref{alg:sa-cq-rttp} was applied to the problem described above in order to reduce the dose in the avoidance structures and elevate it in the target structure, according to the prescription in Table \ref{tab:numerical}. Forty cycles ($N_\mathrm{cycles}=40$) were used and the relaxation parameters, $\lambda_\ell$ and $\gamma_\ell$, were set to their midrange values, $1$ and $1/\Vert|A_\ell\Vert^2$, respectively. Explicitly, $\lambda_1=\lambda_2=\lambda_3=\lambda_4=1$, $\gamma_1 = 1.546\times 10^{-6}$, $\gamma_2 = 1.545\times 10^{-6}$, and $\gamma_3 = \gamma_4 = 1.030\times 10^{-6}$. Figure \ref{fig:dosemaps}(c) shows a visualisation of the dose solution following the algorithmic procedure. It is common in the clinic to evaluate plans using their dose-volume histogram (DVH), which shows the percentage of each structure that has received a certain dose. Figure \ref{fig:DVH} shows a suitable DVH for this plan, with all prescriptions being approximately met. General convergence to the solution is indicated by a decrease in the total number of pixels violating the constraint imposed upon them, shown in the log-loss plot in Figure \ref{fig:total_loss}. Further, log-loss plots for all four types of constraints (minimum dose, maximum dose, lower DVC and upper DVC) are displayed in Figure \ref{fig:loss_detailed}. Again, these all show a general decrease in the number of violations and, therefore, indicate that the solution gradually improves as the number of cycles increases.

\begin{figure}[t!]
\makebox[\textwidth][c]{\includegraphics[width=0.52\columnwidth]{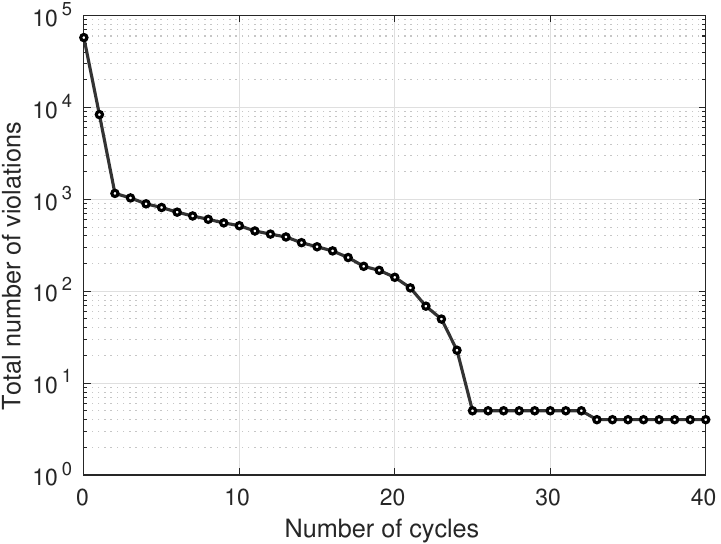}}
\caption{{\footnotesize Number of total violations as a function of number of the algorithmic cycles. A decrease indicates improvement in meeting the prescription.}}\label{fig:total_loss}
\end{figure}

\begin{figure}[t!]
\makebox[\textwidth][c]{\includegraphics[width=0.8\columnwidth]{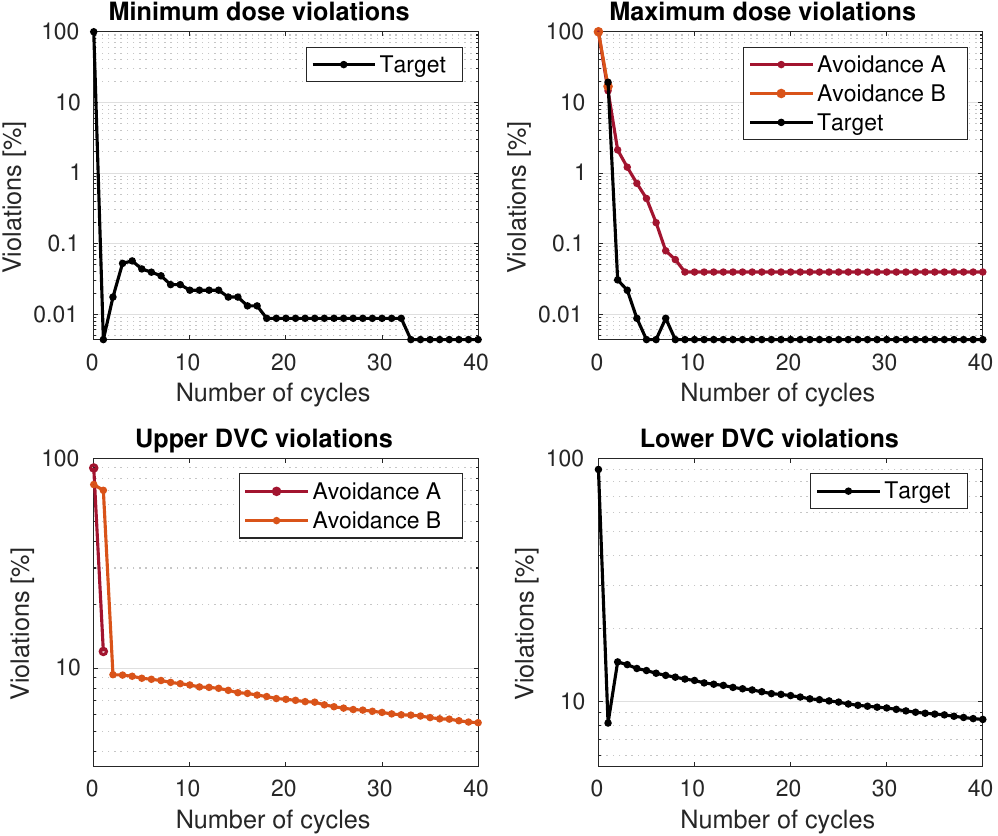}}
\caption{{\footnotesize Percentage of violations as a function of the number of algorithmic cycles, shown separately for HDCs (minimum and maximum doses) and DVCs. An upper DVC is that which is applied to an avoidance structure while a lower DVC is that which is applied to a target structure.}}\label{fig:loss_detailed}
\end{figure}

As mentioned in Section \ref{subsec:structure}, more extensive analysis in the context of radiation therapy treatment planning and, in particular, medical physics is necessary in order to justify the use of the proposed dynamic string-averaging CQ-method. This work is ongoing and will be published in an appropriate medical physics journal.

\section{Conclusions}

We introduced a new split feasibility problem called ``the multiple-operator
split common fixed point problem'' (MOSCFPP). This problem generalizes
some well-known split feasibility problems such as the split convex
feasibility problem, the split common fixed point problem and more.
Following the recent work of \citet{pzcbcs17}, and motivated from
the field of radiation therapy treatment planning, the MOSCFPP involves
additional so-called Percentage Violation Constraints (PVCs) that
give rise to non-convex constraints sets. A new string-averaging CQ
method for solving the problem is introduced, which provides the user
great flexibility in the weighting and order in which the projections
onto the individual sets are executed.

%%% Insert the list of acronyms

\printunsrtglossary[title=List of abbreviations]

% Acknowledgements

\section*{Acknowledgments}

We thank Scott Penfold, Reinhard Schulte and Frank Van den Heuvel
for their help and encouragement of our work on this project. {We
are grateful to the reviewers for their constructive and helpful comments
on the previous version of this paper.} This work was supported by
Cancer Research UK, grant number C2195/A25197, through a CRUK Oxford
Centre DPhil Prize Studentship and by the ISF-NSFC joint research
program grant No. 2874/19. All authors contributed equally to the
writing of this paper. All authors read and approved the final manuscript.

% References
\bibliographystyle{itor} % Use line below if you have the .bib file
%\bibliography{ITOR-refs}

\expandafter\ifx\csname natexlab\endcsname\relax 
\global\long\def\natexlab#1{#1}%
\fi \expandafter\ifx\csname url\endcsname\relax 
\global\long\def\url#1{\texttt{#1}}%
\fi \expandafter\ifx\csname urlprefix\endcsname\relax 
\global\long\def\urlprefix{URL }%
\fi \providecommand{\eprint}[2][]{\url{#2}} %Type = Article

\end{document}